\documentclass[12pt]{amsart}

\usepackage{amscd,amssymb}
\usepackage[all]{xy}
\usepackage[plainpages,backref,urlcolor=blue]{hyperref}
\usepackage[utf8]{inputenc}
\usepackage{mathtools}

\theoremstyle{plain}
\newtheorem{thm}[subsection]{Theorem}
\newtheorem{lem}[subsection]{Lemma}
\newtheorem{prop}[subsection]{Proposition}
\newtheorem{cor}[subsection]{Corollary}

\theoremstyle{definition}
\newtheorem{rk}[subsection]{Remark}

\newtheorem{ex}[subsection]{Example}

\numberwithin{equation}{section}
\newcommand{\OO}{\mathcal O}

\newcommand{\Z}{\mathbb Z}
\newcommand{\Q}{\mathbb Q}

\newcommand{\C}{\mathbb C}
\newcommand{\PP}{\mathbb P}
\newcommand{\N}{\mathbb N}

\DeclareMathOperator{\pd}{pd}
\DeclareMathOperator{\hd}{hd}

\DeclareMathOperator{\codim}{codim}

\DeclareMathOperator{\reg}{reg}

\begin{document}

\title[On the degree of the singular subscheme  of hypersurfaces]{On the degree of the singular subscheme of hypersurfaces in $\PP^n$}

\author[Alexandru Dimca]{Alexandru Dimca$^1$}
\address{Universit\'e C\^ote d'Azur, CNRS, LJAD, France and Simion Stoilow Institute of Mathematics,
P.O. Box 1-764, RO-014700 Bucharest, Romania}
\email{Alexandru.Dimca@univ-cotedazur.fr}

\author[Gabriel Sticlaru]{Gabriel Sticlaru}
\address{Faculty of Mathematics and Informatics,
Ovidius University
Bd. Mamaia 124, 900527 Constanta,
Romania}
\email{gabriel.sticlaru@gmail.com }

\thanks{$^1$ partial support from the project ``Singularities and Applications'' - CF 132/31.07.2023 funded by the European Union - NextGenerationEU - through Romania's National Recovery and Resilience Plan.}



\begin{abstract}
Explicit formulas determining the dimension and the degree of 
the singular subscheme of hypersurfaces in $\PP^n$ are given in terms of the graded Betti numbers of the minimal free resolution of the corresponding Jacobian algebra. This gives in particular new restrictions which must be satisfied by such graded Betti numbers. We define a homologically strictly plus-one generated hypersurface, and show that such a hypersurface has a singular locus of dimension $n-2$ under some conditions. 
\end{abstract}

\maketitle

\section{Introduction}

Let $S=\C[x_0, \ldots, x_n]$ be the polynomial ring in $n+1 \geq 3$ variables $x_0, \ldots, x_n$ with complex coefficients, and let $X:f=0$ be a reduced hypersurface of degree $d\geq 3$ in the complex projective space $\PP^n$. 
We denote by $J_f$ the Jacobian ideal of $f$, i.e. the homogeneous ideal in $S$ spanned by the partial derivatives 
$$f_i=\partial_{x_i}f $$
of $f$,  for $i=0, \ldots , n$,  and  by $M(f)=S/J_f$ the corresponding graded quotient ring, called the Jacobian (or Milnor) algebra of $f$.
Consider the general form of the minimal graded free resolution of the Milnor algebra $M(f)$, which exists by Hilbert Syzygy Theorem, see for instance \cite[Theorem 1.13]{Eis}. One has

\begin{equation}
\label{res2A}
0 \longrightarrow E_n \longrightarrow \ldots 
 \longrightarrow E_2
   \longrightarrow E_1
   \longrightarrow S^{n+1}(1-d)
   \longrightarrow S,
   \end{equation}
where 
\begin{equation}
\label{res2AA}
E_k=\bigoplus_{i_k=1}^{m_k} S(1-d-d_{k,i_k}) 
\end{equation}
for some integers $d_{ij} \geq 0$, with $m_1 \geq n$ and $m_j \geq 0$ for $j \geq 2$.
We assume that 
$$d_{k,1} \leq \ldots \leq d_{k,m_k} \text{ for } k=1, \ldots, n  
 $$ 
and call the ordered sequences of degrees 
$${\bf d}_k=(d_{k,1}, \ldots, d_{k,m_k}) \text{ for } k=1, \ldots, n  
 $$ 
  the {\it graded Betti numbers of the Jacobian algebra} $M(f)$, since they determine and are determined by the usual
graded Betti numbers of the Jacobian algebra $M(f)$ as defined for instance in \cite{Eis}. 

It is known that there is a unique polynomial $P(M(f))(u) \in \Q[u]$, called the {\it Hilbert polynomial} of $M(f)$, and an integer $k_0\in \N$ such that
\begin{equation}
\label{Hpoly}
 \dim M(f)_k= P(M(f))(k)
\end{equation}
for all $k \geq k_0$. We denote by $\Sigma$ the singular subscheme of $X$, which is defined by the Jacobian ideal $J(f)$. The general theory of Hilbert polynomials says that the degree of  $P(M(f))$ is given by the dimension   ${\Sigma}$.
Moreover, if $\dim \Sigma=\delta$, for some $0 \leq \delta \leq n-2$, then the leading term of the polynomial $P(M(f))(k)$ has the form
\begin{equation}
\label{Ppoly}
  \frac{\deg \Sigma}{\delta !}k^{\delta},
\end{equation}
see for instance Theorem 7.5 and the definitions following it in \cite[Chapter 1]{H}. 
In particular, the assumption $\dim \Sigma=0$ implies that the polynomial
$P(M(f))$ is a constant, namely the total Tjurina number of $X$, given by
\begin{equation}
\label{ab0}
P(M(f))=\tau(X)=\sum_{s \in \Sigma} \tau(X,s),
\end{equation}
where $\tau(X,s)$ denotes the Tjurina number of the isolated singularity $(X,s)$.

The  main result of this note is the following computation of the dimension $\dim \Sigma$ and of the degree $\deg \Sigma$ of the singular subscheme $\Sigma$ of $X:f=0$ in terms of the graded Betti numbers of the Jacobian algebra $M(f)$
introduced in \eqref{res2A}. To simplify the statement, we set
\begin{equation}
\label{eqS}
\sigma_j=\sum_{k=1}^n \left((-1)^{k+1}\sum_{i_k=1}^{m_k}d_{k,i_k}^j\right),
\end{equation}
for any integer $0 \leq j \leq n$. Note that
$$\sigma_0=\sum_{k=1}^n (-1)^{k+1}m_k.$$

 \begin{thm}
\label{thm1}
For the minimal resolution \eqref{res2A} of the Jacobian algebra $M(f)$ of a reduced hypersurface $X:f=0$ of degree $d$ in $\PP^n$, one has the following.
\begin{enumerate}

\item 
For any such hypersurface, one has $\sigma_0=n$ and
$\sigma_1=d-1.$

\item One has  $\delta= \dim \Sigma$ for some integer $\delta \in [0,n-2] $ if and only if 
$$ \sigma_j=(-1)^{j+1}(d-1)^j$$
for any integer $1 \leq j < n-\delta$ and $\sigma_{n-\delta} \ne (-1)^{n-\delta+1}(d-1)^{n-\delta}$.

 If this holds, then
$$\deg \Sigma= \frac{1}{(n-\delta)!}\left( (d-1)^{n-\delta}+(-1)^{n-\delta}\sigma_{n-\delta}\right) >0.$$

\item The hypersurface $X$ is smooth if and only if one has
$$m_k=\binom{n+1}{k+1} \text{ and } d_{k,1}=d_{k,2}=\ldots =d_{k,m_k}=k(d-1)$$
for any $k=1, \ldots, n$.

\end{enumerate}

\end{thm}

When the hypersurface $X$ has only isolated singularities, that is when $\dim \Sigma =0$, then using a result by A. du Plessis and C.T.C. Wall quoted below in Theorem \ref{thmC}, we obtain the following new restrictions on the
graded Betti numbers of the Jacobian algebra $M(f)$. See Section 2 for more details.

\begin{cor}
\label{corC} 
If the hypersurface $X$ has  isolated singularities and we set 
$$r=d_{1,1}=\min_{i_1}d_{1,i_1},$$
then one has
$$(n!-1)(d-1)^n-n!r(d-1)^{n-1} \leq  
(-1)^{n}\sigma_n 
\leq (n!-1)(d-1)^n-      n!r(d-r-1)(d-1)^{n-2}.$$
\end{cor}
We say that a reduced hypersurface $X:f=0$ is {\it homologically strictly plus-one generated}, for short HSPOG, if in the resolution
\eqref{res2A} one has $m_1=n+1$, $m_2=1$ and $d_{2,1}=d_{1,i}+1$
for some $1 \leq i \leq n+1$, and $m_k=0$ for any $k \geq 3$.
Strictly plus-one generated hyperplane arrangements have been introduced by Takuro Abe in \cite{Abe}, and the general reduced hypersurfaces which are strictly plus-one generated were defined and studied recently in \cite{Bour}. In particular, the fact that a strictly plus-one generated hypersurface is a HSPOG hypersurface follows from
\cite[Definition 1.1]{Bour}. The advantage of the notion of HSPOG hypersurface is that it does not depend on the ordering of the first order Jacobian syzygies of $f$, as the notion of a strictly plus-one generated hypersurface does, see  \cite[Example 4.8]{Bour}.
It is clear that for any hyperplane arrangement $X$ in $\PP^n$ one has
$\dim \Sigma=n-2$. The next result shows that this property is shared by some HSPOG hypersurfaces of degree not too low.
\begin{cor}
\label{corD} 
If $X$ is a HSPOG hypersurface of degree $d$ in $\PP^n$, then $\dim \Sigma =n-2$ if the following conditions hold
\begin{enumerate}

\item 
either $n=3$ or
$$\max \{d_{1,j }: 1 \leq j \leq n+1\} \leq n(d-2).$$

 \item 
$$d > \frac{n(n+1+\sqrt {n^2-2n-3})}{n+1}.$$
\end{enumerate}
In particular, the inequality (2) holds for $d\geq 4$ when $n=3$, for $d \geq 6$ for $n=4$ and for $d \geq 2n$ for any $n \geq 5$.

\end{cor}

Theorem \ref{thm1}  and Corollary \ref{corC} are proved in Section 2. 
In Section 3 we discuss first some restrictions on the Betti numbers of $M(f)$, with special attention for the case when $X$ has only isolated singularities, see Theorem \ref{thm2}.
 Then we show  that the formula which defines $\deg \Sigma$ in Theorem \ref{thm1} (2) gives an integer $T_0$ under rather general assumptions on the set of integers ${\bf d}_k$, see Proposition \ref{propGS}. On the other hand, the positivity of $T_0$ may fail, as well as the inequality for $T$ given in Corollary \ref{corC}, when 
the integers ${\bf d}_k$ are not the graded Betti numbers of a Jacobian algebra, see Examples \ref{ex1} and \ref{ex2}.
Corollary \ref{corD} is proved in Section 4, where we discuss also some general facts related to the projective dimension $\pd M(f)$ of the Jacobian algebra $M(f)$, see in particular Corollary \ref{cor101}.
This latter result, perhaps well known to the specialists, gives information on the length of the resolution \eqref{res2A} when we know the dimension $\dim \Sigma$. This information is particularly effective when $\dim \Sigma$ is small, and says that the length of the resolution \eqref{res2A} must be quite large in such a case.

\section{Proof of Theorem \ref{thm1} and of Corollary \ref{corC}}

We start by introducing some polynomials $A_j \in \Z[a]$. For any integer $a$ and $k > |a|$, one sets
\begin{equation}
\label{eq1}
n! \dim S_{k+a}=n! \binom{k+a+n}{n}=\sum_{j=0}^{n}A_j(a)k^j.
\end{equation}
One has the following.
\begin{lem}
\label{lem1}
For any $0 \leq j \leq n$, the polynomial $A_j$ has degree $n-j$ and its leading coefficient is
$\binom{n}{n-j}$.
\end{lem}
\proof
It is enough to note the equality
$$A_j(a)=\sum_I \prod_{i \in I} (a+i),$$
where $I$ runs through all the subset with $n-j$ elements of the set
$\{1,2, \ldots, n\}$.
\endproof
Using the resolution \eqref{res2A}, we get the following
\begin{equation}
\label{res2C1}
 \dim M(f)_{s+d-1}= \dim S_{s+d-1}-(n+1) \dim S_s+\sum_{k=1}^n (-1)^{k-1}
\left ( \sum_{i_k=1}^{m_k}\dim S_{s-d_{k,i_k}}\right )
\end{equation}
for any $s$ large enough.
Using now \eqref{eq1}, we get the following.

\begin{lem}
\label{lem2}
With the above notation, for any large integer $s$, one has the equality
 $$n! \dim M(f)_{s+d-1}=\sum_{j=0}^{n}B_js^j$$
 where
 $$B_j= A_j(d-1)-(n+1)A_j(0)+\sum_{k=1}^n (-1)^{k-1}\left ( \sum_{i_k=1}^{m_k}A_j(-d_{k,i_k})\right).
$$
 \end{lem}
Using these direct computations, we can prove Theorem \ref{thm1} as follows.
By definition of the Hilbert polynomial $P(M(f))$ we have
$$n!P(M(f))(s+d-1)=n! \dim M(f)_{s+d-1}.$$
Since the hypersurface $X$ is reduced, we have
$$\deg P(M(f))=\dim \Sigma \leq n-2,$$
and hence the coefficients of $s^n$ and of $s^{n-1}$ in Lemma \ref{lem2} must vanish. 
The coefficient of $s^n$ is clearly
$$B_n=1-(n+1)+ \sum_{k=1}^n (-1)^{k-1}m_k$$
and this proves the first part of claim (1).
To prove the remaining part, we note that the polynomial
$A_{n-1}(a)$ has the form
$$A_{n-1}(a)=na+b,$$
where $b \in \Z$ according to Lemma \ref{lem1}. It follows that one has
$$B_{n-1}=n(d-1)+b-(n+1)b+\sum_{k=1}^n (-1)^{k-1}\sum_{i_k=1}^{m_k}(-nd_{k,i_k}+b)=$$
$$=n\left(d-1-b+ \sum_{k=1}^n (-1)^{k}(\sum_{i_k=1}^{m_k}d_{k,i_k})+b\right).$$
This proves the second part of claim (1).

To prove the claim (2), we note that $\delta=\dim \Sigma$ implies that
$B_j=0$ for all $j>\delta$. Using Lemma \ref{lem1} we may write 
$$A_j(a)=c_{j,n-j}a^{n-j} + \ldots + c_{j,1}a+c_{j,0}$$
where 
$$c_{j,n-j}=\binom{n}{n-j}.$$
Consider first the case 
 $\delta=n-2$, the maximal possible value. Then the first claim in (2) is in fact the second part of claim (1), so there is nothing to prove.
To prove the second claim in (2), note that by Lemma \ref{lem2} we have
$$B_{n-2}=(c_{n-2,2}(d-1)^2+c_{n-2,1}(d-1)+c_{n-2,0})-(n+1)c_{n-2,0}+ $$
$$+\sum_{k=1}^n (-1)^{k-1}\sum_{i_k=1}^{m_k}(c_{n-2,2}d_{k,i_k}^2-c_{n-2,1}d_{k,i_k}+c_{n-2,0})=$$
$$=c_{n-2,2}\left( (d-1)^2+\sum_{k=1}^n (-1)^{k-1}\sum_{i_k=1}^{m_k}d_{k,i_k}^2 \right)+ c_{n-2,1}\left(d-1+ \sum_{k=1}^n (-1)^{k}\sum_{i_k=1}^{m_k}d_{k,i_k} \right)+$$
$$+c_{n-2,0}\left(-n
+\sum_{k=1}^n (-1)^{k-1}m_k\right).$$
The sums in the last two brackets vanish according to claim (1).
Since 
$$c_{n-2,2}= \binom{n}{n-2}= \frac{n!}{2!(n-2)!}$$
by Lemma \ref{lem1}, the claim follows using \eqref{Ppoly}.

Assume now that $\delta \leq n-3$ and hence the coefficient $B_{n-2}$ vanishes.
 So
the first bracket in the above formula for $B_{n-2}$ has to vanish as well. This gives us the first claim in (2),  that is $\sigma_2=-(d-1)^2$. As above, a simple computation using these vanishings shows that
$$B_{n-3}=c_{n-3,3}\left( (d-1)^3+\sum_{k=1}^n (-1)^{k}\sum_{i_k=1}^{m_k}d_{k,i_k}^3 \right).$$
Using this we complete the case $\delta=n-3$ as above.
By  decreasing induction on $\delta$ one proves the claim (2) for any
$\delta$ exactly as above, using Lemma \ref{lem1} and \eqref{Ppoly} for the second claim in (2).

To prove (3), we assume first that $X$ is smooth. Then the partial derivatives $f_0, \ldots,f_n$ form a regular sequence, and the resolution
\eqref{res2A} is just the corresponding Koszul complex. This clearly implies that 
$$m_k=\binom{n+1}{k+1} \text{ and } d_{k,1}=d_{k,2}=\ldots =d_{k,m_k}=k(d-1)$$
for any $k=1, \ldots, n$. Conversely, assume that these equalities hold for the graded Betti numbers of $X$. Using the claim (2), we have to show all the equalities
$$ \sum_{k=1}^n (-1)^{k+j}(\sum_{i_k=1}^{m_k}d_{k,i_k}^j)=(d-1)^j$$
for  $1 \leq j \leq n$.
Dividing by $(-1)^{j-1}(d-1)^j$, these equalities become $\sigma_j'=(-1)^{j-1}$,
where 
\begin{equation}
\label{eq10}
\sigma_j'=\sum_{k=1}^n (-1)^{k+1}\binom{n+1}{k+1}k^j
\end{equation}
 for  $1 \leq j \leq n$.
To check these equalities, we start with the obvious equality 
\begin{equation}
\label{eq11}
\sum_{i=0}^{n+1} (-1)^{i}\binom{n+1}{i}t^i=1-(n+1)t+\sum_{k=1}^{n} (-1)^{k+1}\binom{n+1}{k+1}t^{k+1}=(1-t)^{n+1}.
\end{equation}
If we set $t=1$ in this equality, we get exactly the first equality in the claim (1), that is 
$$\sigma_0=\sigma_0'=n.$$ 
Now we take the derivative with respect to $t$ in \eqref{eq11}
and we get
\begin{equation}
\label{eq12}
-(n+1)+\sum_{k=1}^{n} (-1)^{k+1}\binom{n+1}{k+1}(k+1)t^{k}=-(n+1)(1-t)^{n}.
\end{equation}
If we set $t=1$ in \eqref{eq12}, we get 
$$-(n+1)+\sum_{k=1}^{n} (-1)^{k+1}\binom{n+1}{k+1}(k+1)=-(n+1)+\sigma_1' +\sigma_0'=0.$$
This gives $\sigma_1'=1$. 
To get $\sigma_2'=-1$, we multiply the equality \eqref{eq12} by $t$, then take again the derivative with respect to $t$ and set $t=1$.
Using $(k+1)^2=k^2+2k+1$, this gives 
$$-(n+1)+\sigma_2'+2 \sigma_1'+\sigma_0'=0,$$
which yields $ \sigma_2'=-1$.
Continuing in this way, we get all the relations $\sigma'_j=(-1)^{j-1}$ for
$1 \leq j \leq n$ and conclude that the singular subscheme $\Sigma$ is empty.
This completes the proof of Theorem \ref{thm1}.




One has the following result, see \cite[Theorem 5.3]{duPCTC01}.
\begin{thm}
\label{thmC}
If the hypersurface $X$ has at most isolated singularities and $d_{1,1}=\min_{i_1}d_{1,i_1}$,
then
$$(d-1)^n-d_{1,1}(d-1)^{n-1} \leq \tau(X) \leq (d-1)^n-d_{1,1}(d-d_{1,1}-1)(d-1)^{n-2}.$$
\end{thm}
On the other hand, Theorem \ref{thm1} (2) for $\delta=0$ implies that
\begin{equation}
\label{eq13}
n! \tau(X)= (d-1)^{n}+\sum_{k=1}^n (-1)^{k+n-1}(\sum_{i_k=1}^{m_k}d_{k,i_k}^{n}) .
\end{equation}
To prove Corollary \ref{corC} it is enough to multiply the inequalities in Theorem \ref{thmC} by $n!$, replace $n! \tau(X)$  by the above value of it and finally simplify one term $(d-1)^n$.

\begin{rk}
\label{rkC} 
The lower bound in Theorem \ref{thmC} is attained for any pair $(d,d_{1,1})$.
Indeed, it is enough to find a degree $d$ reduced curve $C: f'(x_0,x_1,x_2)=0$ such that $d_{1,1}$ is the minimal exponent of $C$ and
$$\tau(C)=(d-d_{1,1}-1)(d-1),$$
and then take $X:f=0$, with 
$$f=f'(x_0,x_1,x_2)+x_3^d+ \ldots + x_n^d.$$
The existence of curves $C$ as above is shown in \cite[Example 4.5]{3syz} and a complete characterization of them is given in \cite[Theorem 3.5 (1)]{3syz}.
\end{rk}

\section{Various restrictions on the  graded Betti numbers of Jacobian algebras}
 
As we have already noted in the case of surfaces in $\PP^3$, see \cite[Remark 3.7]{Betti3}, the graded Betti numbers ${\bf d}_k$ for $k=1,\dots,n$ satisfy certain inequalities. To state them, recall that the sequences
${\bf d}_k$ are ordered, that is
$$d_{k,1} \leq d_{k,2} \leq \ldots \leq d_{k,m_k},$$
when $m_k \geq 2$.
Note first that $m_1 \geq n$, with equality if and only if $X:f=0$ is a free hypersurface. Then, in full generality, one clearly has
\begin{equation}
\label{eqR1}
 d_{1,1} \geq 0,
\end{equation}
with equality if and only if $X$ is a cone over a hypersurface in $\PP^{n-1}$.
If $m_2>0$, that is if $X$ is not a free hypersurface, a secondary syzygy must involve at least three primary syzygies, see \cite[Remark 3.7]{Betti3}.
Using  the minimality of the resolution \eqref{res2A}, we deduce in exactly the same way that in general
if $m_j>0$ for some $j$ satisfying $2 \leq j \leq n$, then $m_{j-1} \geq 3$
and
\begin{equation}
\label{eqR2}
 d_{j,1}=\min _{i_j} d_{j,i_j}>\max\{d_{j-1,1},d_{j-1,2},d_{j-1,3}\}.
\end{equation}
When the hypersurface $X$ has at most isolated singularities, we also have
stronger restrictions on the corresponding Betti numbers.
We recall that the Castelnuovo-Mumford regularity $\reg M(f)$ is defined as follows in terms of the minimal resolution
\eqref{res2A}
\begin{equation}
\label{eqREG}
\reg M(f)= \max_{k=1,n}(d_{k,i_{m_k}}+d-1-k-1),
\end{equation}
see \cite[Section 4A]{Eis}, since our $E_k$ corresponds to $F_{k+1}$ in that definition. For the next result, see \cite[Proposition 2.3]{Hess}.
 \begin{thm}
\label{thm2}
The Castelnuovo-Mumford regularity of the Jacobian algebra $M(f)$ of a  singular hypersurface $X:f=0$ of degree $d\geq 3 $ in $\PP^n$ having only isolated singularities satisfies
$$\reg M(f) \leq  (n+1)(d-2)-1.$$
In other words, the Betti numbers coming from the minimal resolution
\eqref{res2A} satisfy the inequalities
$$(I_k) \ :  \ d_{k,i_{m_k}} \leq n(d-2)+k-1$$
for any $k=1,\ldots,n$.
\end{thm}

We have shown that the inequalities in Corollary \ref{corC} are not an arithmetical consequence of the relations given in \eqref{eqR1}, \eqref{eqR2} and in Theorem \ref{thm2} combined with the equalities given in Theorem \ref{thm1}, claims (1) and (2) when $n=2$ or $n=3$. Indeed, the case $n=2$ is discussed in \cite[Remark 5.3]{Betti2} and the case $n=3$ is discussed in \cite[Remark 3.7]{Betti3}. 
We conjecture that a similar situation occurs for any $n \geq 2$, and this shows the interest in Corollary \ref{corC}.

A simpler remark is that a set of {\it potential } graded Betti numbers ${\bf d}_k $ satisfying all the above conditions  in \eqref{eqR1}, \eqref{eqR2} and Theorem \ref{thm2} and satisfying also the claims in Theorem \ref{thm1} (1) and the first claim in Theorem \ref{thm1} (2) fails to be the {\it actual} Betti numbers of a hypersurface $X$ with $\dim \Sigma = \delta$ if the rational number given by the second claim in Theorem \ref{thm1} (2) is either negative, or is not an integer. 
However, this last situation cannot occur, since one has the following result.
\begin{prop}
\label{propGS}
Assume that an integer $d \geq 3$ and a set of  integers 
$${\bf d}_k=(d_{k,1}, \ldots, d_{k,m_k}) \text{ for } k=1, \ldots, n,  
 $$ 
 not  necessarily coming from the minimal resolution of a Jacobian algebra $M(f)$, 
satisfy the relations 
$$\sigma_j=(-1)^{j+1}(d-1)^j$$
for any $1 \leq j <t $, with $t \leq n$ an integer, where $\sigma_j$ is defined as in \eqref{eqS}. Then the integer
$$N_t=(d-1)^t+(-1)^t\sigma_t$$ is divisible by $t!$.
\end{prop}
\proof
Consider the polynomial
$$P(z)=z(z+1)\ldots (z+t-1)=z^t+c_1z^{t-1}+\ldots+ c_{t-1}z.$$
For any integer $z$, the integer $P(z)$ is divisible by $t!$, as being the product of $t$ consecutive integers. It follows that
$$\sigma_t+c_1\sigma_{t-1}+ \ldots + c_{t-1}\sigma_1$$
is divisible by $t!$. Using the relations satisfied by $\sigma_j$ for $1 \leq j <t$, we get that
$\sigma_t+R$, where 
$$R=(-1)^{t}c_1(d-1)^{t-1} + \ldots +c_{t-1}(d-1)$$
is divisible by $t!$. Note that
$$-P(-(d-1))=(-1)^{t+1}(d-1)^t+R$$
is also divisible by $t!$, and hence the difference
$$\sigma_t+R -((-1)^{t+1}(d-1)^t+R)=\sigma_t+(-1)^t(d-1)^t$$ is divisible by $t!$.
Multiplying by $(-1)^t$ we get  our claim.
\endproof
We conclude this section by some examples of potential graded Betti numbers
related to  3-folds in $\PP^4$ with isolated singularities.
These numbers satisfy some of the above conditions, but not all of them, so they are not the actual Betti numbers of a 3-fold $X$ with isolated singularities.
\begin{ex}
\label{ex1}
Consider the following set of potential graded Betti numbers
$${\bf d}_1=(2_9,3), {\bf d}_2=(4_7,5_3), {\bf d}_3=(6_2,7_3) \text{ and } {\bf d}_4=(9).$$
Hence $m_1=m_2=10$, $m_3=5$ and $m_4=1$.
For $d=3$  and $\delta=0$, they satisfy the conditions given in \eqref{eqR1}, \eqref{eqR2} and  the equalities from Theorem \ref{thm1}, claims (1) and (2), but not the positivity condition from Theorem \ref{thm1}, (2). Indeed, a direct computation using the formula given in the second claim in Theorem \ref{thm1} (2) yields
$$\tau(X)=\deg(\Sigma)=-8$$
for any 3-fold $X$ in $\PP^4$ having these graded Betti numbers.
In addition, the inequalities $(I_3)$ and $(I_4)$ from
Theorem \ref{thm2} are not satisfied.
 Therefore we have several obstructions to the existence of such a 3-fold.
\end{ex}

 \begin{ex}
\label{ex2}
Consider the following set of potential graded Betti numbers
$${\bf d}_1=(2_{10},10_{17},14_{17}), {\bf d}_2=(4_{10},11_{68}), {\bf d}_3=(6_5,12_{102}) \text{ and } {\bf d}_4=(8,13_{68}).$$
Hence $m_1=44$, $m_2=78$, $m_3=107$ and $m_4=69$.
For $d=3$  and $\delta=0$, they satisfy the same conditions as in Example \ref{ex1}. In this case all the inequalities $(I_k)$ for
$k=1,2,3,4$ fail.
A direct computation  yields
$$\sigma_4=392.$$
For any 3-fold $X$ in $\PP^4$ having these graded Betti numbers Corollary \ref{corC} yields
$$-16 \leq \sigma_4 \leq 368.$$
Hence we have again two obstructions to the existence of such a 3-fold.
Note that in this case the value of $\tau(X)$ given by Theorem \ref{thm1}
is positive, hence we need either Corollary \ref{corC} or Theorem \ref{thm2} to conclude.
\end{ex}

\section{Some generalities and the proof of Corollary \ref{corD}}
Recall that if a graded $S$-module $M$ has a minimal free resolution
$$0 \to F_p \to \ldots \to F_1 \to F_0,$$
where $F_p \ne 0$, then one says that the {\it projective dimension} $\pd M$ of $M$ is equal to $p$, see for instance \cite{Eis}. Sometimes this invariant is called the {\it homological dimension} and it is denoted by $\hd M$, as for instance in \cite{RT}. If we compare this definition with our resolution \eqref{res2A} it follows that we have
\begin{equation}
\label{eqE100}
\pd M(f) = \max \{k \ : \ 1 \leq k \leq n \text{ such that } m_k>0\}+1.
\end{equation}
In particular, $\pd M(f) \leq n+1$.
Let ${\bf p}_i$ for $i=1, \ldots,q$ be the associated prime ideals of the $S$-module $M(f)$. Then it is known that
\begin{equation}
\label{eqE101}
\pd M(f) \geq \codim {\bf p}_i,
\end{equation}
for $i=1, \ldots,q$, where the codimension is taken with respect to the polynomial ring $S$, or equivalently, in the ambient affine space $\C^{n+1}$, see  \cite[Corollary A.2.16]{Eis} or even better \cite[Exercise 19.8]{Eis0}. The minimal associated prime ideals of $M(f)$ correspond exactly to the irreducible components of the support $|\Sigma|$ of the singular scheme $\Sigma$, if this scheme is non empty. In this way we get the following.
\begin{cor}
\label{cor101}
If the hypersurface $X:f=0$ is singular, then 
$$\pd M(f) \geq \codim \Sigma$$
or, equivalently
$$\dim \Sigma \geq n-1-\max \{k \ : \ 1 \leq k \leq n \text{ such that } m_k>0\}.$$
In particular, if $\dim \Sigma=0$, then either $E_n \ne 0$ or $E_{n-1} \ne 0$ in the resolution \eqref{res2A}.
\end{cor}
In general the inequality in Corollary \ref{cor101} is rather weak. For instance, if $X$ is a hyperplane arrangement having at least two hyperplanes, then
$\codim \Sigma=2$. On the other hand, if the Jacobian ideal $J_f$ is not saturated, then the maximal ideal
$${\it m}=(x_0,\ldots,x_n) \subset S$$
is associated with $M(f)$ as noted in \cite[page 2 of Introduction]{MNS}, and hence in this case
$$n+1\geq \pd M(f)\geq \codim {\bf m}=n+1,$$
and hence $\pd M(f)=n+1$, the maximal possible value.
This equality happens in particular for a generic hyperplane arrangement, see \cite[Theorem 4.5.3]{RT}.

If $X:f=0$ is a HSPOG hypersurface in $\PP^n$, then obviously by
\eqref{eqE100} one has $\pd M(f)=3$, and hence Corollary \ref{cor101} implies that $\dim \Sigma \geq n-3$. It follows that Corollary \ref{corD} gives a stronger claim in this situation under some additional mild restrictions on $d=\deg f$ and on the maximal value of the exponents $d_{1,i}$.

To prove Corollary \ref{corD}, we simplify the notation by setting $d_j=d_{1,j}$ for all $j=1,\ldots, n+1$ and we do no longer assume that
$$d_1 \leq d_2 \leq \ldots \leq d_{n+1}.$$
With this notation, we may assume with no loss of generality that $d_{2,1}=d_{n+1}+1$. 
These are in fact the notations used in \cite[Theorem 4.5]{Bour}. Using this
result, or our Theorem \ref{thm1} (1) we see that
\begin{equation}
\label{eqE10}
d_1+\ldots +d_n=d.
\end{equation}
Using Theorem \ref{thm1} (2), we see that $\dim \Sigma=n-2$ if and only if $A>0$, where
\begin{equation}
\label{eqE11}
A=(d-1)^2+\sigma_2=(d-1)^2+\sum_{j=1}^{n+1}d_j^2-(d_{n+1}+1)^2= d^2-2d+\sum_{i=1}^{n}d_i^2-2d_{n+1}.
\end{equation}
Hence $A>0$ is equivalent to
$$d^2+\sum_{i=1}^{n}d_i^2-2d>2d_{n+1}.$$
If $n=3$, then either $\dim\Sigma =1$ and there is nothing to prove, or
$\dim \Sigma =0$ and then $d_{n+1} \leq 6(d-2)$ by condition $(I_1)$ in
Theorem \ref{thm2}.
Therefore, using our  assumption (1), we get that
it is enough in any of the two cases to show that
$$d^2+\sum_{i=1}^{n}d_i^2-2d>2n(d-2).$$
Finally, using the well known inequality
$$n\sum_{i=1}^{n}d_i^2 \geq (\sum_{i=1}^{n}d_i)^2=d^2,$$
it follows that $A>0$ follows from the inequality
\begin{equation}
\label{eqE13}
g(d)=(n+1)d^2-2n(n+1)d+4n^2>0.
\end{equation}
The discriminant of this second degree polynomial in $d$ is
$$\Delta=4n^2((n+1)^2-4(n+1))=4n^2(n^2-2n-3).$$
For $n=3$ we get $\Delta=0$ and $g(d)=4(d-3)^2$, hence the inequality \eqref{eqE13} holds for any
 $d >3$.
 For $d \geq 4$, we see that the largest root $d'$ of the equation $g(d)=0$ satisfies
 $$d' =\frac{n(n+1+\sqrt {n^2-2n-3})}{n+1} < \frac{n(n+1+(n-1))}{n+1} =\frac{2n^2}{n+1}.$$
 Since the inequality \eqref{eqE13} holds for any
 $d>d'$, this proves our claims in Corollary \ref{corD}. Indeed, for $n=4$ we get the approximative value $d'=5.78$.



\end{document}